\documentclass[reqno]{amsart}
\newtheorem{Theorem}{Theorem}[section]
\newtheorem{Definition}[Theorem]{Definition}

\newtheorem{Lemma}[Theorem]{Lemma}
\newtheorem{Proposition}[Theorem]{Proposition}

\begin{document}
%[[[ Title
\title[Lifespan est. for semirelativistic eqn. with critical nonlinearity]%
{Note on the lifespan estimate of solutions
for non-gauge invariant semilinear massless semirelativistic equations
with some scaling critical nonlinearity}

\author[K. Fujiwara]{Kazumasa Fujiwara}
\address{Graduate School of Mathematics, Nagoya University,
Furocho, Chikusaku, Nagoya, 464-8602, Japan}
\email{fujiwara.kazumasa@math.nagoya-u.ac.jp}
\keywords{semirelativistic equation, lifespan estimate, non-gauge invariant power-type nonlinearity}
\subjclass[2010]{35Q40, 35A01}
\date{}

\maketitle

\begin{abstract}
In this manuscript,
in the $L^1$ scaling critical case,
a lifespan estimate of solutions to the Cauchy problem for non-gauge invariant semilinear semirelativistic equations is considered.
The lifespan estimate is given by the modified test function method with a fractional Laplace operator.
The main obstacle to obtaining the lifespan estimate is the non-locality of the fractional Laplace operator.
To treat the non-locality, special test functions are introduced.
\end{abstract}
%]]]

%[[[ \section{Introduction}
\section{Introduction}
%[[[ In this manuscript,
In this manuscript,
we consider the following Cauchy problem for (massless) semilinear semirelativistic equations:
	\begin{align}
	\begin{cases}
	i \partial_t u + (-\Delta)^{1/2} u = |u|^p,
	&t \in \lbrack 0,T) , \quad x \in \mathbb R^n,\\
	u(0)=u_0,
	&x \in \mathbb R^n,
	\end{cases}
	\label{eq:1}
	\end{align}
where $p > 1$ and $n$ is a positive integer.
The fractional Laplace operator $(-\Delta)^{1/2}$ is defined by a Fourier multiplier
with symbol $|\xi|$: $(-\Delta)^{1/2} = \mathcal F^{-1} |\xi| \mathcal F$,
where $\mathcal F$ denotes the usual Fourier transform.
The aim of this manuscript is to show some a priori lifespan estimates of
non-global weak solutions with the scaling critical power $p=(n+1)/n$ for some initial data.
%]]]

%[[[ Semirelativistic equations are also known as half-wave or fractional Schr\"odinger equations.
Semirelativistic equations are also known as half-wave or fractional Schr\"odinger equations.
Indeed, a semirelativistic equation is derived by a formal factorization of a wave equation.
On the other hand,
it is known that semirelativistic and wave equations do not share the property of solutions.
Specifically, the differential operator $(-\Delta)^{1/2}$ is non-local and
makes it difficult to apply the analysis of the classical wave and Schr\"odinger equations.
%]]]

%[[[ It may be directly seen that $(-\Delta)^{1/2}$ has the same scaling property as $\nabla$
It may be directly seen that $(-\Delta)^{1/2}$ has the same scaling property as $\nabla$.
Namely, the identity
	\[
	(-\Delta)^{1/2} (f(\lambda \cdot)) (x) = \lambda (-\Delta)^{1/2} f(\lambda x)
	\]
holds for any positive constant $\lambda$.
Therefore, $\lambda^{1/(p-1)} u (\lambda t, \lambda x)$ satisfies \eqref{eq:1} for any $\lambda >0$
with initial data $\lambda^{1/(p-1)} u_0 (\lambda x)$ as long as $(u,u_0)$ satisfies \eqref{eq:1}.
For $q \geq 1$, we call $p=(n+q)/n$ a critical power in the $L^q(\mathbb R^n)$ framework
because the identities
	\[
	\|\lambda^{1/(p-1)} u_0(\lambda \cdot)\|_{L^q(\mathbb R^n)}
	= \|\lambda^{n/q} u_0(\lambda \cdot)\|_{L^q(\mathbb R^n)}
	= \|u_0 \|_{L^q(\mathbb R^n)}
	\]
hold for any $\lambda >0$.
We also say that the power $p$ is subcritical if $p$ is less than the corresponding critical power.
By the analogy of heat, damped wave, and Schr\"odinger equations,
it is expected that
there exist no global weak solutions to \eqref{eq:1} for some integrable initial data
when $p \leq (n+1)/n$: in the $L^1(\mathbb R^n)$ scaling critical and subcritical cases.
For details of the nonexistence results for the equations above,
we refer the reader to the examples in \cite{IW2013,IS2019,MP1998,Z1999}.
Here, the global weak solutions to \eqref{eq:1} are defined as follows:
\begin{Definition}
Let $u_0 \in L^1 \cap L^2(\mathbb R^n)$.
For $T > 0$, we say that $u$ is a weak solution to \eqref{eq:1} on $[0,T)$,
if $u$ belongs to
$L_\mathrm{loc}^1(0, T ; L^2(\mathbb R^n)) \cap L_\mathrm{loc}^1(0, T ; L^p(\mathbb R^n))$
and the following identity
	\begin{align}
	&\int_0^\infty \int_{\mathbb R^n} u(t,x)
	\overline{(i \partial_t + (- \Delta)^{1/2}) \phi (t,x) } dx \thinspace dt
	\nonumber\\
	&= i \int_{\mathbb R^n} u_0(x) \overline{\phi (0,x)} dx + \int_0^\infty \int_{\mathbb R^n} |u(t,x)|^p \overline{ \phi(t,x)} dx \thinspace dt
	\label{eq:2}
	\end{align}
holds for any $\phi \in C^\infty(\mathbb R^{n+1})$
satisfying $\phi, \partial_t \phi, (-\Delta)^{1/2} \phi \in L^1 \cap L^\infty(\mathbb R^{n+1})$ and
	\[
	\mathrm{supp}\thinspace \phi \subset (-\infty,T\rbrack \times \mathbb R^n.
	\]
Moreover, we define the maximal existence time $T_m=T_m(u_0)$ as
	\[
	T_m
	= \sup \{T > 0 \ ; \ \mbox{There is a weak solution for \eqref{eq:1} on $[0,T)$.}\}.
	\]
We also say that $u$ is a global weak solution to \eqref{eq:1} when $T_m=\infty$.
\end{Definition}
%]]]

%[[[ Indeed, in \cite{FO2015},
Indeed, in \cite{FO2015},
it is shown that when $n=1$ in the $L^1(\mathbb R)$ scaling critical and subcritical cases,
there are no global solutions with some initial data as follows:
%[[[ proposition : Fujiwara Ozawa
\begin{Proposition}
\label{Proposition:1.2}
If $n=1$, $1 < p \leq 2$, and $u_0 \in L^1 \cap L^2 (\mathbb R)$ satisfying that
	\begin{align}
	\mathrm{Re} (u_0) \equiv 0,
	\quad
	- \mathrm{Im} \bigg( \int_{\mathbb R} u_0(x) dx \bigg) > 0,
	\label{eq:3}
	\end{align}
then there are no global weak solutions;
namely, there are no weak solutions on $[0,T)$ for $T$ big enough.
\end{Proposition}
We note that classical test function methods are not applicable to \eqref{eq:1}
because of the non-local operator $(-\Delta)^{1/2}$.
Especially, the lack of the (pointwise) Leibniz rule for $(-\Delta)^{1/2}$ is the main obstacle.
For details of classical test function methods,
we refer the readers to \cite{MP1998,Z1999}.
In \cite{FO2015}, to avoid the difficulty arising from $(-\Delta)^{1/2}$,
\eqref{eq:1} is reduced to the equation
	\begin{eqnarray}
	(\partial_t^2 - \Delta) \mathrm{Im} \thinspace u = - \partial_t(|u|^p).
	\label{eq:4}
	\end{eqnarray}
Equation \eqref{eq:4} is obtained by applying the conjugate operator $- (i \partial_t - (-\Delta)^{1/2})$
to both sides of \eqref{eq:1} and taking the imaginary part of the resulting equation.
We remark that
this reduction can be regarded as the inverse operation of the derivation of semirelativistic equations
from wave equations.
Later, Inui \cite{I2016} obtained not only nonexistence results
but also lifespan estimates of weak solutions to \eqref{eq:1} in the $L^2(\mathbb R^n)$ subcritical case with $n \geq 1$
by improving the approach of \cite{FO2015}.
However, we remark that the method of \cite{I2016} is not applicable to either $L^1(\mathbb R^n)$ or $L^2(\mathbb R^n)$ scaling critical cases.
We also remark that the condition \eqref{eq:3} is technical
and to the best of our knowledge,
the precise condition of initial data for the nonexistence of global weak solutions is not known.
For a related topic, we refer the reader to \cite{FGp}.
%]]]
%]]]

%[[[ In \cite{F2018},
In \cite{F2018},
in the $L^1(\mathbb R^n)$ scaling subcritical case with $n \geq 1$,
the following lifespan estimate is obtained:
%[[[ Proposition : Nonexistence,
\begin{Proposition}[{\cite[Proposition 4]{F2018}}]
\label{Proposition:1.3}
Let $n \geq 1$ and $1 < p < (n+1)/n$.
Let $f \in L^1(\mathbb R^n) \backslash \{0\}$ satisfy
	\begin{align}
	\mathrm{Re} \thinspace f(x) = 0,
	\quad \mathrm{and} \quad
	- \mathrm{Im} \thinspace f(x) > 0
	\label{eq:5}
	\end{align}
for any $x \in \mathbb R^n$.
For sufficiently small $\varepsilon$ positive,
there are no global weak solutions to \eqref{eq:1} with $u_0 = \varepsilon f$
and the lifespan estimate
	\[
	T_m \leq C \varepsilon^{\frac{1}{n-1/(p-1)}}
	\]
holds with a positive constant $C=C(n,p,f)$ independent of $\varepsilon$.
\end{Proposition}
%]]]
Here we denote by $C(\ast, \cdots, \ast)$ a constant depending on the quantities appearing in parenthesis.
We remark that Proposition \ref{Proposition:1.3} is natural from the viewpoint of the scaling property of \eqref{eq:1}.
For example, similar lifespan estimates for Schr\"odinger equations are found in \cite{II2015,IS2019}.
Proposition \ref{Proposition:1.3} is shown without the reduction of \eqref{eq:4}
but with the pointwise control of fractional derivatives of specific test functions.
We note that C\'ordoba and C\'ordoba \cite{CC2004} showed that the pointwise estimate
	\begin{align}
	(-\Delta)^{s/2} (\phi^2 )(x) \leq 2 \phi(x) ((-\Delta)^{s/2} \phi)(x)
	\label{eq:6}
	\end{align}
holds for any $0 \leq s \leq 2$, $\phi \in \mathcal S(\mathbb R^2)$, and $x \in \mathbb R^n$
but \eqref{eq:6} is insufficient to apply the test function method
unless weak solutions are non-negative.
For details, see \cite{DR2014}.
On the other hand, in \cite{F2018},
the following estimate is introduced and plays a critical role in the proof for Proposition \ref{Proposition:1.3}:
%[[[ Lemma : Fractional Derivative
\begin{Lemma}
\label{Lemma:1.4}
For $q > 0$,
there exists a positive constant $C = C(n,q)$ such that
the estimate
	\begin{align*}
	&|( (-\Delta)^{1/2} (1 + \cdot^2)^{-q/2} ) (x) |\\
	&\leq C
	\begin{cases}
	(1+x^2)^{-(q+1)/2},
	&\quad \mathrm{if} \quad 0 < q < n,\\
	(1+x^2)^{-(n+1)/2} (1+\log (1+|x|)),
	&\quad \mathrm{if} \quad q = n,\\
	(1+x^2)^{-(n+1)/2},
	&\quad \mathrm{if} \quad q > n.
	\end{cases}
	\end{align*}
holds for any $x \in \mathbb R^n$.
\end{Lemma}
%]]]
Especially, in the case of Lemma \ref{Lemma:1.4},
$(-\Delta)^{1/2}$ can be treated like a classical differential operator in the scaling subcritical case.
For a related topic, we also refer the reader to \cite{DR2021}.
%]]]

%[[[ In the scaling critical case,
In the scaling critical case,
a more careful treatment is required and Lemma \ref{Lemma:1.4} is not sufficient for the blowup analysis.
In \cite{DF2021},
the nonexistence of global weak solutions to generalized \eqref{eq:1} in the $L^1(\mathbb R^n)$ scaling critical case with $n \geq 1$ is shown
by using the following identity of the fractional derivative of specific test functions:
\begin{Lemma}[{\cite[Corollary 3.3]{DF2021}}]
\label{Lemma:1.5}
For $q > n$,
	\[
	(-\Delta)^\sigma [ ( 1 + \cdot^2)^{-q/2}](0)
	= 2^{2 \sigma} \frac{\Gamma(\sigma+n/2)}{\Gamma(n/2)} \frac{\Gamma(\sigma+q/2)}{\Gamma(q/2)},
	\]
where $\Gamma$ denotes the usual gamma function.
\end{Lemma}
We note that $\Gamma(1/2+\cdot)/\Gamma(\cdot)$ is an increasing function on $\lbrack 0, \infty)$.
Then we set
	\[
	\eta_0(x) = ( 1 + x^2)^{-(n+1)/2} - C_0 ( 1 + x^2)^{-(n+2)/2},
	\]
where the constant $C_0 \in (0,1)$ is given and computed by
	\begin{align}
	C_0
	&=\frac{\Gamma((n+2)/2)^2}{\Gamma((n+1)/2) \Gamma((n+3)/2)}
	\nonumber\\
	&= \Big( \frac \pi 2 \Big)^{2 ( n \thinspace \mathrm{mod} \thinspace 2) - 1}\frac{1}{(n+1)} \Big( \frac{n!!}{(n-1)!!} \Big)^2
	\label{eq:7}
	\end{align}
where $n!!$ and $(n-1)!!$ denote the double factorials of $n$ and $n-1$, respectively.
In order to show the nonexistence of global weak solutions in the scaling critical case,
it is important that $\eta_0$ is a positive smooth function and $(-\Delta)^{1/2} \eta_0(0)=0$.
We will use these properties later.
However, the lifespan estimate of weak solutions cannot be obtained by the method of \cite{DF2021},
just as one cannot obtain any lifespan estimate for Schr\"odinger equations with classical test function methods.
%]]]

%[[[ The purpose of the current manuscript is to obtain the lifespan estimate
The purpose of the current manuscript is to obtain the lifespan estimate
of weak solutions with $p=(n+1)/n$: in the $L^1(\mathbb R^n)$ scaling critical case.
In particular,
we combine the approaches of \cite{IS2019} and \cite{DF2021}.
In \cite{IS2019},
Ikeda and Sobajima derived an ordinal differential inequality (ODI)
from damped wave equations with respect to a scaling parameter,
while an ODI with respect to time was used in \cite{F2018}.
This ODI with respect to the scaling parameter plays a critical role
for obtaining the lifespan estimate of weak solutions in the scaling critical case.
However, the derivation of the ODI relies on the classical Leibniz rule
and one cannot apply the approach of \cite{IS2019} directly to the blowup analysis of \eqref{eq:1}.
In this manuscript,
we modify the argument of \cite{IS2019} with the idea of \cite{DF2021}
so as to deal with the fractional differential operator $(-\Delta)^{1/2}$.
%]]]

%[[[ Now
Now we can provide our main statement.
\begin{Theorem}
\label{Theorem:1.6}
Let $n \geq 1$ and $p = (n+1)/n$.
Let $f \in L^1(\mathbb R^n)$ satisfy \eqref{eq:5}.
For sufficiently small $\varepsilon$ positive,
there are no global weak solutions to \eqref{eq:1} with $u_0 = \varepsilon f$
and the following lifespan estimate holds with a positive constant $C=C(n,p,f)$:
	\[
	T_m \leq \exp(C \varepsilon^{-1/n}).
	\]
\end{Theorem}
We remark that Theorem \ref{Theorem:1.6} is again natural from the viewpoint of the scaling property of \eqref{eq:1}.
Indeed, Theorem \ref{Theorem:1.6} corresponds to Proposition 2.1 of \cite{IS2019}.
However, we note that, in general, Theorem \ref{Theorem:1.6} is not a sharp estimate.
For example, in the one dimensional case,
if initial data are (possibly integrable but) singular at the origin in some sense,
t (possibly integrable but)hen it is shown in \cite{F2015} that there are no local weak solutions to \eqref{eq:1};
namely, there exist no weak solutions on $[0,T)$ for any positive $T$.
In addition, even if the initial data are bounded,
then a sharper lifespan estimate
	\begin{align}
	T_m \leq C \varepsilon^{-p+1}
	\label{eq:8}
	\end{align}
is obtained for any $p>1$ in \cite{Fp}.
Indeed, in the $L^1(\mathbb R)$ critical case, \eqref{eq:8} is rewritten by
	\[
	T_m \leq C \varepsilon^{-1}.
	\]
Roughly speaking,
this is because when $n=1$, \eqref{eq:1} can be identified with the semilinear advection equation:
	\begin{align}
	\begin{cases}
	\partial_t w + \partial_x w = w^p,
	& t \in \lbrack 0, T), \quad x \in \mathbb R,\\
	w(0,x) = w_0(x),
	&x \in \mathbb R,
	\end{cases}
	\label{eq:9}
	\end{align}
whose solutions are formally but explicitly given by
	\[
	w(t,x) = \Big( w_0(t-x)^{-p+1} - (p-1) t \Big)^{-1/(p-1)},
	\]
with real valued $w_0$.
This representation of $w$ formally implies
the nonexistence of local weak solutions to \eqref{eq:1} and the lifespan estimate \eqref{eq:8}.
On the other hand, in the multi-dimensional case,
\eqref{eq:1} cannot be identified with \eqref{eq:9}
and \eqref{eq:8} seems not to hold.
Therefore,
it is unclear whether or not Theorem \ref{Theorem:1.6} is sharp.
%]]]

%[[[ In next section,
In the next section,
we collect some estimates of fractional derivatives.
In the last section,
we show Theorem \ref{Theorem:1.6}.
%]]]
%]]]

%[[[ \section{Some Estimates of fractional derivatives}
\section{Some Estimates of Fractional Derivatives}
%[[[ we set
We set
	\begin{align}
	\eta(t,x)
	&= ( 1 + t^2+x^2)^{-(n+1)/2} - C_0 ( 1 + t^2 + x^2)^{-(n+2)/2},
	\label{eq:10}
	\end{align}
where $C_0 \in (0,1)$ is given by \eqref{eq:7}.
The purpose of this section is to show the following estimate:
\begin{Proposition}
\label{Proposition:2.1}
The estimate
	\begin{align}
	|(-\Delta)^{1/2} \eta(t,x)|
	&\leq C \min \Big\{ (t^2+x^2)^{1/2} , ( 1 + t^2+x^2)^{-(n+1)/2} \Big\}
	\label{eq:11}
	\end{align}
holds for any $t \in \mathbb R$ and $x \in \mathbb R^n$ with some positive constant $C$ independent of $t$ and $x$.
\end{Proposition}
%]]]

%[[[ We, at fist, recall that in \cite{DF2021}, a generalization of Lemma \ref{Lemma:1.4} is shown:
First, we recall that in \cite{DF2021}, a generalization of Lemma \ref{Lemma:1.4} is shown:
\begin{Lemma}[{\cite[Lemma 3.2]{DF2021}}]
\label{Lemma:2.2}
Let $f \in C^2(\mathbb R^n)$ be almost decreasing with almost decreasing second derivatives:
	\[
	|f(y)| \leq C_1 |f(x)|,\quad
	\sup_{|\alpha|=2} |\partial^\alpha f(y)|
	\leq C_1 \sup_{|\alpha|=2} |\partial^\alpha f(x)|
	\]
hold when $|x| \leq |y|$ with some positive constant $C_1$ independent of $x$ and $y$,
then the pointwise estimate
	\begin{align*}
	| (-\Delta)^{\sigma} f (x)|
	&\leq C {|x|^{-n-2 \sigma}} \int_{|y| < 3 |x|} |f(y)| dy
	+ |f(x)| |x|^{-2 \sigma}\\
	&+ C \frac{2^{3-2 \sigma}}{2-2 \sigma}
	|x|^{2-2 \sigma}
	\sum_{|\alpha|=2} \frac{|\alpha|}{\alpha!}
	|\partial^\alpha f ( \frac x 2 ) |,
	\end{align*}
holds for any $\sigma\in(0,1)$ and $|x| > 1$ with some positive constant $C = C(n,C_1)$.
\end{Lemma}
Lemma \ref{Lemma:2.2} implies that we have
	\[
	|(-\Delta)^{1/2} \eta(t,x)| \leq C (1+x^2)^{-(n+1)/2}
	\]
for $|x|>1$ with $C$ independent of $t$ and $x$.
Therefore \eqref{eq:11} holds when $\max(t,1) \leq |x|$ because the estimate
	\[
	(1+x^2)^{-1}
	\leq 2 (1+t^2+x^2)^{-1}
	\]
holds.
In order to show \eqref{eq:11} when $\max(1, |x|) \leq t$,
we need the following lemma:
%]]]

%[[[ Lemma: decay
\begin{Lemma}
\label{Lemma:2.3}
The estimate
	\[
	|(-\Delta)^{1/2}[(1+t^2+\cdot^2)^{-q/2}](x)|
	\leq C \log(4 t) (1+t^2+x^2)^{-(q+1)/2}
	\]
holds for any $q > 0$ and $t \geq \max(|x|,1)$
with some positive constant $C$ independent of $t$ and $x$.
\end{Lemma}

\begin{proof}
Let $f(x) = (1+t^2+x^2)^{-q/2}$.
We note that we have
	\begin{align*}
	(-\Delta)^{1/2} f(x)
	&= C \int_{\mathbb R^n} \frac{f(x+y) - 2f(x) + f(x-y)}{|y|^{n+1}} d y\\
	&= C \int_{|y|>2t} \frac{f(x+y) - 2f(x) + f(x-y)}{|y|^{n+1}} d y\\
	&+ C \int_{|y|<2t} \frac{f(x+y) - 2f(x) + f(x-y)}{|y|^{n+1}} d y.
	\end{align*}
For the representation of the fractional derivative above,
we refer the reader to \cite{NPV2012}.
The first integral on the RHS of the second identity above is estimated by
	\begin{align*}
	&\int_{|y|>2t} \frac{f(x+y) - 2f(x) + f(x-y)}{|y|^{n+1}} d y\\
	&\leq C f(x) \int_{|y|>2t} |y|^{-n-1} d y\\
	&\leq C (1+t^2+x^2)^{-(q+1)/2}.
	\end{align*}
Here we have used the fact that the estimate
	\begin{align}
	t^{-2}
	\leq 3 (1+t^2+x^2)^{-1}
	\label{eq:12}
	\end{align}
holds when $t \geq \max(|x|, 1)$.
So as to estimate the second integral,
we compute it with the Taylor theorem by
	\begin{align}
	& \int_{|y|<2t} \frac{f(x+y) - 2f(x) + f(x-y)}{|y|^{n+1}} d y
	\nonumber\\
	&= C \sum_{|\alpha|=2} \frac{2}{\alpha!} \int_{|y|<2t} \frac{y^\alpha}{|y|^{n+1}} \int_0^1 (1-\theta) \partial_x^\alpha f(x+\theta y) d \theta d y
	\nonumber\\
	&+ C \sum_{|\alpha|=2} \frac{2}{\alpha!} \int_{|y|<2t} \frac{y^\alpha}{|y|^{n+1}} \int_0^1 (1-\theta) \partial_x^\alpha f(x-\theta y) d \theta d y.
	\label{eq:13}
	\end{align}
The first and second integrals on the RHS of \eqref{eq:13} are estimated similarly.
When $|x| < 1$,
the RHS of \eqref{eq:13} is estimated by
	\[
	C (1+t^2)^{-q/2-1} \int_{|y|<2t} |y|^{1-n} dy
	\leq C (1+t^2+x^2)^{-(q+1)/2}.
	\]
When $|x| > 1$,
we divide the first integral on the RHS of \eqref{eq:13} into two cases where $0 < \theta < |x|/4t$ and $|x|/4t < \theta < 1$.
In the first case, since $|x+\theta y| > |x|/2$, we estimate
	\begin{align*}
	&\bigg| \int_{|y|<2t} \frac{y^\alpha}{|y|^{n+1}} \int_0^{|x|/4t} (1-\theta) \partial_x^\alpha f(x+\theta y) d \theta \thinspace d y \bigg|\\
	&\leq C (1+t^2+|x|^2)^{-q/2-1} \frac{|x|}{t} \int_{|y|< 2t} |y|^{1-n} d y \\
	&\leq C (1+t^2+|x|^2)^{-(q+1)/2}.
	\end{align*}
By changing the integral variable as $z=x+\theta y$,
we estimate the second case as
	\begin{align*}
	&\bigg| \int_{|y|<2t} \frac{y^\alpha}{|y|^{n+1}} \int_{|x|/4t}^1 (1-\theta) \partial_x^\alpha f(x+\theta y) d \theta d y \bigg|\\
	&\leq \int_{|x|/4t}^1 \theta^{-1} d \theta \int_{\mathbb R^n} |z-x|^{-n + 1} (1+t^2+|z|^2)^{-q/2-1} d z \\
	&\leq \log(4 t) \int_{\mathbb R^n} |z-x|^{-n + 1} (1+t^2+|z|^2)^{-q/2-1} d z.
	\end{align*}
Furthermore, we estimate
	\begin{align*}
	&\int_{\mathbb R^n} |z-x|^{-n + 1} (1+t^2+|z|^2)^{-q/2-1} d z\\
	&\leq (1+t^2)^{-q/2-1} \int_{|z|<2|x|} |z-x|^{-n + 1} d z
	+ \int_{|z| > 2|x|} |z|^{-n + 1} (1+t^2+|z|^2)^{-q/2-1} d z\\
	&\leq C (1+t^2)^{-q/2-1} |x|
	+ C (1+t^2)^{-q/2-1} \int_{2|x| < r < 2t} d r
	+ C \int_{r > 2t} (1+t+r)^{-q-2} d r\\
	&\leq C (1+t^2+x^2)^{-(q+1)/2},
	\end{align*}
where we have used \eqref{eq:12} again.
\end{proof}
%]]]

%[[[ \begin{proof}[Proof of Proposition \ref{Proposition:2.1}]
\begin{proof}[Proof of Proposition \ref{Proposition:2.1}]
Lemmas \ref{Lemma:2.2} and \ref{Lemma:2.3} imply that
the estimate
	\[
	|(-\Delta)^{1/2} \eta(t,x)|
	\leq C (1+t^2+x^2)^{-(n+1)/2}
	\]
holds for any $t,x$ satisfying $\max(t,|x|)>1$ with some positive constant $C$ independent of $t$ and $x$.
Then it is sufficient to show \eqref{eq:11} when $\max(t,|x|)<1$.
Lemma \ref{Lemma:1.5} implies that the identity
	\[
	(-\Delta)^{1/2} \eta(0,0) = 0
	\]
holds.
Moreover, it is also seen that we have
	\[
	\| \partial_t (-\Delta)^{1/2} \eta \|_{L^\infty(\mathbb R^{n+1})}
	+ \| \nabla (-\Delta)^{1/2} \eta \|_{L^\infty(\mathbb R^{n+1})}
	< \infty.
	\]
Therefore, the mean value theorem implies that \eqref{eq:10} holds when $t^2+x^2<1$.
\end{proof}
%]]]
%]]]

%[[[ \section{Proof}
\section{Proof of Theorem \ref{Theorem:1.6}}
We consider the weak form \eqref{eq:2} with the test function
	\[
	\phi_r(t,x) = \rho \bigg( \frac{t}{r+1} \bigg) \eta \bigg( \frac{t}{r+1}, \frac{x}{r+1} \bigg),
	\]
where $\rho : \mathbb R \to [0,1]$ is a smooth decreasing function satisfying that
	\[
	\rho(\tau) = \begin{cases}
	1,
	&\mathrm{if} \quad \tau \leq 1/2,\\
	0,
	&\mathrm{if} \quad \tau \geq 1.
	\end{cases}
	\]
and $\rho'(\tau) \leq C \rho(\tau)^{n/(n+1)}$ for any $\tau \in \mathbb R$.
We remark that such $\rho$ can be constructed by using bump functions.
Here we put
	\[
	\psi_r(t,x)
	= \min\bigg\{ \bigg( \frac{t^2+x^2}{(r+1)^2} \bigg)^{(n+1)/2n} , \bigg( \frac{t^2+x^2}{(r+1)^2} \bigg)^{-1/4n} \bigg\}
	\phi_r(t,x).
	\]
Proposition \ref{Proposition:2.1} implies that the estimates
	\begin{align}
	&| - i \partial_t \phi_r(t,x) + (-\Delta)^{1/2} \phi_r(t,x)|
	\nonumber\\
	&\leq \frac{C}{r+1} \min\bigg\{ \bigg(\frac{t^2+x^2}{(r+1)^2}\bigg)^{1/2}, \bigg( 1 + \frac{t^2+x^2}{(r+1)^2} \bigg)^{-(n+1)/2} \bigg\}
	\rho \bigg( \frac{t}{r+1} \bigg)^{n/(n+1)}
	\nonumber\\
	&\leq \frac{C}{r+1}
	\bigg( 1 + \frac{t^2+x^2}{(r+1)^2} \bigg)^{-(n+1/2)/2(n+1)} \psi_r(t,x)^{n/(n+1)} \chi_r(t)
	\label{eq:14}
	\end{align}
hold, where cut-off function $\chi_r$ is given by
	\[
	\chi_r(t)
	= \begin{cases}
	1 &\mathrm{if} \quad t < r+1,\\
	0 &\mathrm{if} \quad t \geq r+1.
	\end{cases}
	\]
Here we note that the estimates
	\[
	(1-C_0) (1+t^2+x^2)^{-(n+1)/2}
	\leq \eta(t,x)
	\leq (1+t^2+x^2)^{-(n+1)/2}
	\]
follow from \eqref{eq:10}.
By \eqref{eq:2}, we compute
	\begin{align}
	&i \int_{\mathbb R^n} u_0(x) dx + \int_0^T \int_{\mathbb R^n} |u(t,x)|^{(n+1)/n} \phi_r(t,x) dx dt
	\nonumber\\
	&= \int_0^T \int_{\mathbb R^n} u(t,x) \bigg( - i \partial_t \phi_r(t,x) + (-\Delta)^{1/2} \phi_r(t,x) \bigg) dx \thinspace dt.
	\label{eq:15}
	\end{align}
Combining \eqref{eq:14} and the H\"older estimate,
the RHS of \eqref{eq:15} is estimated by
	\begin{align*}
	&\bigg| \int_0^T \int_{\mathbb R^n} u(t,x) \bigg( - i \partial_t \phi_r(t,x) + (-\Delta)^{1/2} \phi_r(t,x) \bigg) dx \thinspace dt \bigg|\\
	&\leq C \frac{1}{r+1} \bigg( \int_0^{r+1} \int_{\mathbb R^n}
		\bigg( 1 + \frac{t^2+x^2}{(r+1)^2} \bigg)^{-(n+1/2)/2} dx \thinspace dt \bigg)^{1/(n+1)}\\
	&\cdot \bigg( \int_0^T \int_{\mathbb R^n} |u(t,x)|^{(n+1)/n} \psi_r(t,x) dx \thinspace dt \bigg)^{n/(n+1)}\\
	&\leq C \bigg( \int_0^T \int_{\mathbb R^n} |u(t,x)|^{(n+1)/n} \psi_r(t,x) dx \thinspace dt \bigg)^{n/(n+1)}.
	\end{align*}
We set
	\begin{align*}
	y(r) &= \int_0^T \int_{\mathbb R^n} |u(t,x)|^{(n+1)/n} \psi_r(t,x) dx \thinspace dt,\\
	Y(R) &= \int_0^R \frac{y(r)}{r+1} dr.
	\end{align*}
The Fubini theorem implies that $Y$ is rewritten by
	\begin{align}
	Y(R)
	= \int_0^T \int_{\mathbb R^n} |u(t,x)|^{(n+1)/n} \int_0^R \frac{\psi_r(t,x)}{r+1} dr.
	\label{eq:16}
	\end{align}
Since $\phi_r(t,x)$ is increasing with respect to $r$ for any fixed $t$ and $x$,
the estimate
	\begin{align}
	\int_0^R \frac{\psi_r(t,x)}{r+1} dr
	\leq C \phi_R(t,x)
	\label{eq:17}
	\end{align}
holds for any $t$ and $x$ with some positive constant $C$ independent of $t$ and $x$.
Indeed, we have
	\begin{align*}
	&\int_0^R \frac{\psi_r(t,x)}{r+1} dr\\
	&\leq \phi_R(t,x)
	\int_0^\infty \min\bigg\{ \bigg( \frac{t^2+x^2}{(r+1)^2} \bigg)^{1/2}, \bigg( \frac{t^2+x^2}{(r+1)^2} \bigg)^{-1/4} \bigg\}
	\frac{1}{r+1} dr\\
	&\leq \phi_R(t,x) \int_0^\infty \min\{ r'^{-1/2}, r'^{-3/2}\} dr'
	\end{align*}
with $r' = (t^2+x^2)^{1/2}/(r+1)$.
Combining \eqref{eq:16} and \eqref{eq:17},
the estimate
	\begin{align}
	Y(R) \leq C \int_0^T \int_{\mathbb R^n} |u(t,x)|^{(n+1)/n} \phi_R(t,x) dr
	\label{eq:18}
	\end{align}
holds.
\eqref{eq:15} and \eqref{eq:18} imply that the estimate
	\[
	\varepsilon + Y(R) \leq C ((R+1)Y'(R))^{n/(n+1)}
	\]
holds with some positive constant $C=C(n,p,f)$ independent of $R$.
Therefore we have
	\begin{align*}
	Y(R)
	\geq \Big( \varepsilon^{-1/n} - C \log(R+1) \Big)^{-n} - \varepsilon.
	\end{align*}
Since the estimate above and $Y(R) < \infty$ hold for any $R \in (0,T_m)$,
we get
	\[
	T_m \leq \exp(C \varepsilon^{-1/n}).
	\]
%]]]

%[[[ \section*{Acknowledgment}
\section*{Acknowledgment}
The author is supported in part by JSPS Grant-in-Aid for Early-Career Scientists No. 20K14337.
%]]]

%[[[ \begin{thebibliography}{00}

%]]]


\begin{thebibliography}{00}
\bibitem{CC2004}
A. C{\'o}rdoba and D. C{\'o}rdoba,
A maximum principle applied to quasi-geostrophic equations,
Comm. Math. Phys.,
\textbf{249}(2004),
511--528.

\bibitem{DF2021}
M. D’Abbicco and K. Fujiwara,
A test function method for evolution equations with fractional powers of the Laplace operator,
Nonlinear Anal. Theory, Methods Appl.,
\textbf{202} (2021) 112114,
https://doi.org/10.1016/j.na.2020.112114.

\bibitem{DR2014}
M. D'Abbicco and M. Reissig,
Semilinear structural damped waves,
Math. Methods Appl. Sci.,
\textbf{37} (2014),
1570--1592.

\bibitem{DR2021}
T.A. Dao and M. Reissig,
Blow-up results for semi-linear structurally damped $\sigma$-evolution equations,
Springer INdAM Ser.,
\textbf{43} (2021),
213 - 245.,
\verb+https://doi.org/10.1007/978-3-030-61346-4_10+.

\bibitem{NPV2012}
E. Di Nezza, G. Palatucci, E. Valdinoci,
\emph{Hitchhiker’s guide to the fractional Sobolev spaces},
Bull. Sci. Math. \textbf{136}(2012), 521--573.

\bibitem{F2015}
K. Fujiwara,
Remark on local solvability of the Cauchy problem for semirelativistic equations,
J. Math. Anal. Appl.,
\textbf{432}(2015),
744 - 748,
https://doi.org/10.1016/j.jmaa.2015.07.009.

\bibitem{F2018}
K. Fujiwara,
A note for the global nonexistence of semirelativistic equations with nongauge invariant power type nonlinearity,
Math. Methods Appl. Sci.,
\textbf{41} (2018),
4955 - 4966,
https://doi.org/10.1002/mma.4944.

\bibitem{Fp}
K. Fujiwara,
Lifespan estimates of 1D non-gauge invariant semilinear semirelativistic equations,
to appear in Appl. Math. Lett.

\bibitem{FGp}
K. Fujiwara and V. Georgiev,
On global existence of $L^2$ solutions for 1D periodic NLS with quadratic nonlinearity,
to appear in J. Math. Phys.

\bibitem{FO2015}
K. Fujiwara and T. Ozawa,
Remarks on global solutions to the {C}auchy problem for semirelativistic equations with power type nonlinearity,
Int. J. Math. Anal.,
\textbf{9} (2015),
2599 - 2610,
https://doi.org/10.12988/ijma.2015.58211.

\bibitem{IW2013}
M. Ikeda and Y. Wakasugi,
Small data blow-up of $L^2$-solution for the nonlinear {S}chr\"odinger equation without gauge invariance,
Differ. Integral Equ.,
\textbf{26}(2013),
1275 - 1285.

\bibitem{II2015}
M. Ikeda and T. Inui,
Some non-existence results for the semilinear {S}chr\"odinger equation without gauge invariance,
J. Math. Anal. Appl.,
\textbf{425} (2015),
758 - 773,
https://doi.org/10.1016/j.jmaa.2015.01.003.

\bibitem{IS2019}
M. Ikeda and M. Sobajima,
Sharp upper bound for lifespan of solutions to some critical semilinear parabolic,
dispersive and hyperbolic equations via a test function method,
Nonlinear Anal. Theory, Methods Appl.,
\textbf{182} (2019),
57 - 74,
https://doi.org/10.1016/j.na.2018.12.009.

\bibitem{I2016}
T. Inui,
Some nonexistence results for a semirelativistic {S}chr\"odinger equation with nongauge power type nonlinearity,
Proc. Am. Math. Soc.,
\textbf{144} (2016),
2901 - 2909,
https://doi.org/10.1090/proc/12938.

\bibitem{MP1998}
E. Mitidieri, S. I. Pohozaev,
\emph{The absence of Global Positive Solutions to Quasilinear Elliptic Inequalities},
Doklady Mathematics \textbf{57} (1998), 250--253.
	

\bibitem{Z1999}
Q. S. Zhang,
Blow-up results for nonlinear parabolic equations on manifolds,
Duke Mathematical Journal, 
\textbf{97}(1999),
109--114,
515–-539,
https://doi.org/10.1215/S0012-7094-99-09719-3.

\end{thebibliography}
\end{document}